\documentclass[12pt]{amsart}
\usepackage{amssymb}
\usepackage{amscd}

\newtheorem{theorem}{Theorem}[section]
\newtheorem{corollary}[theorem]{Corollary}
\newtheorem{lemme}[theorem]{Lemma}
\newtheorem{defin}{Definition}[section]

\allowdisplaybreaks \theoremstyle{definition}
\theoremstyle{remark} \numberwithin{equation}{section}

\begin{document}

\title [Uncertainty principles for the Clifford-Fourier transform]{ Uncertainty principles for  the  Clifford-Fourier transform\\}%

\author[ J. El Kamel, R. Jday]{ Jamel El Kamel \qquad and \qquad Rim Jday }
 \address{Jamel El Kamel. Department of Mathematics fsm. Monastir 5000, Tunisia.}
 \email{jamel.elkamel@fsm.rnu.tn}
 \address{Rim Jday. Department of Mathematics fsm. Monastir 5000, Tunisia.}
 \email{rimjday@live.fr}
 \begin{abstract} In this paper,  we estabish  an analogue of Hardy's theorem and Miyachi's theorem for the Clifford-Fourier transform.
\end{abstract}
\maketitle
%%%%%%%%%%%%%%%%%%%%%%%%%%%%%%%%%%%%%%%%%%%%%%%%%%%%%%%
{\it keywords:} Clifford algebra, Clifford-Fourier transform, Heat kernel, Uncertainty principle, Hardy's theorem, Miyachi's theorem.\\

\noindent MSC (2010) 42B10, 30G35.\\
\section{Introduction}
Initially, uncertainty principles were given by two theorems: Heisenberg's inequality \cite{key-19}  and Hardy's theorem which ensured that is impossible for a nonzero   function and its Fourier transform  to be simultaneously small.
Afterwards, many theorems were devoted to describe this smallness for example Miyachi \cite{key-13}, Cowling and Price \cite{key-4}, Donoho and Stark  \cite{key-8}, Benedicks-Amrein-Berthier \cite{key-1}  and  Beurling \cite{key-18}.\\ 
 Hardy's theorem \cite{key-9}  for the classical Fourier transform  states : If we  suppose $p$ and $q$ be positive constants and $f$ be a  measurable function on the real line satisfying   $|f(x)|\leq C e^{-px^2}$ and $|\mathcal{F}(f)(\lambda)|\leq C e^{-q\lambda^2} $ for some positive
 constant C, then (i)  $f=0$ if $pq>\frac{1}{4}$; (ii) $f=A e^{-px^2}$ for some constant $A$ if $pq=\frac{1}{4};$ (iii) there are many $f$ if $pq<\frac{1}{4}.$\\
 Miyachi's theorem  \cite{key-13} asserts that if $f$ is a measurable function on $\mathbb{R}$ satisfying  : 
 $$e^{ax^2}f \in L^1(\mathbb{R})+L^\infty(\mathbb{R})$$
and 
$$\displaystyle\int_{\mathbb{R}}log^+\frac{|e^{\frac{y^2}{4a}}\hat{f}(y)|}{\lambda}dy <\infty,$$
for some positive constants $a$ and $\lambda$, then $f$ is a constant multiple of $e^{-ax^2}$.\\

 In Dunkl analysis, analogues of Heisenberg's inequality \cite{key-15,key-16}, Hardy's \cite{key-16}, Miyachi's \cite{key-3}, Cowling and Price's   \cite{key-12}, Donoho and Stark's \cite{key-12} and Beurling's \cite{key-12} theorems were established.\\
The last decades, there has been an increasing interest in the study of the Clifford-Fourier transform on $\mathbb{R}^m$ introduced in \cite{key-2} and studied in \cite{key-6,key-7}.\\
\noindent In 2010, H. De Bie and Y. Xu in \cite{key-7} wrote the Clifford-Fourier transform as an integral transform 

$$\mathcal{F}_{\pm}(f)(y)={(2\pi)}^\frac{-m}{2}\int_{\mathbb{R}^m} K_{\pm}(x,y)f(x)dx,$$
where the kernel function $K_{\pm}(x,y)$ was given by an explicit expression. Recently, H. De Bie and Y. Xu in \cite{key-6} show that the Clifford-Fourier transform is a continuous operator on Schwartz class functions. Besides, they define a translation operator related to the Clifford-Fourier transform  and introduce a convolution structure based on translation operator. In the even dimension case, they give an inversion formula for the Clifford-Fourier transform.\\
 \noindent Heisenberg's inequality \cite{key-5,key-20}, Donoho and Stark's theorem and Benedicks's theorem \cite{key-11} are obtained for the Clifford-Fourier transform.\\
 The purpose of this paper  is to generalize Hardy's theorem  and Miyachi's theorem for the Clifford-Fourier transform on $\mathbb{R}^m$.\\ 
 Our  paper is organized as follows. In section 2, we present basic notions and notations related to the Clifford algebra. 
  In section 3, we recall some results for the Clifford-Fourier transform that will be useful in the sequel. Also, 
  we establish some new properties associated  to the kernel of the  Clifford-Fourier transform as well as for the  Clifford-Fourier transform. 
 In section 4, Clifford-heat kernel  is introduced and studied. In section 5, we provide  Hardy's theorem for the  Clifford-Fourier transform on $\mathbb{R}^m$  when $m$ even. 
 Section 6 is devoted to Miyachi's theorem for the  Clifford-Fourier transform when $m$ even.\\ 
\noindent Throughout this paper, the letter C indicates a positive constant that is not necessarily the same in each occurrence.
\section{Notations and preliminaries}
 The Clifford algebra $Cl_{0,m}$  over $\mathbb{R}^m$ is  a non commutative algebra generated by the basis $\{e_1,..,e_m\}$ satisfying the rules :\\

 \begin{equation}
 \left\{
\begin{array}{cc}
 \displaystyle  e_ie_j=-e_je_i,\qquad\qquad \text{if}\quad i\neq j; \qquad\\
   \displaystyle e_i^2=-1,\qquad\qquad\quad \forall 1\leq i\leq m. \\

\end{array}
\right.
 \end{equation} 
This algebra can be decomposed as 
  \begin{equation}
 Cl_{0,m}=\oplus_{k=0}^m Cl_{0,m}^k,
  \end{equation}
 with $Cl_{0,m}^k$ the space of vectors defined by
  \begin{equation}
 Cl_{0,m}^k=span\{e_{i_1}..e_{i_k},i_1<..<i_k\}.
  \end{equation}
 Hence, $\{1,e_1,e_2,..,e_{12},..,e_{12..m}\}$ forms a basis of $Cl_{0,m}.$ \\
 A Clifford number  $x$ in $Cl_{0,m}$ is written as follows :
  \begin{equation}
 x=\displaystyle\sum_{A\in J}e_Ax_A,
  \end{equation} 
 where $J :=\{0,1,..,m,12,..,12..m\},$ $x_A$ is a real number  and $e_A$ belongs to the  basis of $Cl_{0,m}$ defined above.\\
 The norm of such element $x$ is given by :
  \begin{equation}
 {||x||}_c={\left(\displaystyle\sum_{A\in J} x_A^2\right)}^\frac{1}{2}.
  \end{equation}
  In particular, if $x$ is a vector in $Cl_{0,m}$, then
  \begin{equation}
  ||x||_c^2=-x^2.
  \end{equation}
 The Clifford-Dirac operator and Clifford-Laplace operator  are  defined respectively \\ by :
  \begin{equation}
  \partial_x=\displaystyle\sum_{i=1}^m e_i\partial_{x_i},
 \end{equation}
   and
    \begin{equation}
  \Delta_x=\displaystyle\sum_{i=1}^m \partial_{x_i}^2.
   \end{equation}
    We have  the following relation :
   \begin{equation}
  \Delta_x=-\partial_{x}^2.
   \end{equation}
 We introduce respectively the Clifford-Gamma operator  associated to a vector $x$, the inner product and the wedge product of two vectors $ x$ and $y$ :
  \begin{equation}
 \qquad\qquad \Gamma_x:=-\displaystyle\sum_{j<k}e_je_k(x_j\partial_{x_k}-x_k\partial_{x_j});
  \end{equation}
   \begin{equation}
 \quad <x,y> :=\displaystyle\sum_{j=1}^m x_jy_j=-\frac{1}{2}(xy+yx);
  \end{equation}
   \begin{equation}
\qquad \qquad \qquad\quad  x\wedge y:=\displaystyle\sum_{j<k}e_je_k(x_j y_k-x_ky_j)=\frac{1}{2}(xy-yx).
 \end{equation}
 \noindent In the sequel, we consider  functions  defined on $\mathbb{R}^m$ and taking values in $Cl_{0,m}$.  Such functions can be decomposed as :
  \begin{equation}
 f(x)=f_0(x)+\displaystyle\sum_{i=1}^me_if_i(x)+\displaystyle\sum_{i<j}e_ie_jf_{ij}(x)+..+e_1..e_mf_{1..m}(x),
  \end{equation}
 with $f_0,f_i,..,f_{1..m}$ all real-valued functions.\\
Let us recall some functional spaces :\\ 
$\bullet S(\mathbb{R}^{m})\otimes Cl_{0,m}$ the Schwartz space of infinitely differentiable functions on $\mathbb{R}^{m}$ and taking values in $Cl_{0,m} $ which are rapidly decreasing as their derivatives,\\
 $\bullet \mathcal{P}_k$ the space of homogenious polynomials of degree $k$ taking values in $Cl_{0,m}$,\\
% $\bullet \mathcal{P}$ the space of polynomials taking values in $Cl_{0,m}$, i.e\\
%$$ \mathcal{P} := \mathbb{R}[x_1, . . . , x_m] \otimes Cl_{0,m}.$$
$\bullet \mathcal{M}_k:=\text{ker}$ $ \partial_{x}\cap \mathcal{P}_k$ the space of spherical monogenics of degree $k$,\\
$\bullet L^p(\mathbb{R}^m)\otimes Cl_{0,m}$ the space of integrable functions  taking values in $Cl_{0,m}$  equipped with the norm 
 \begin{equation}
||f||_{p,c}=\left(\displaystyle\int_{\mathbb{R}^m}{||f(x)||}_c^pdx\right)^{\frac{1}{p}}
=\left(\displaystyle\int_{\mathbb{R}^m}\left(\displaystyle\sum_{A\in J}(f_A(x))^2\right)^\frac{p}{2}dx\right)^{\frac{1}{p}},
 \end{equation}
where $J=\{0,1,..,m,12,13,23..,12..m\},$\\
$\bullet B(\mathbb{R}^m)\otimes Cl_{0,m}$ a class of integrable functions taking values in $Cl_{0,m}$ and satisfying
 \begin{equation}
||f||_{B,c}:=\int_{\mathbb{R}^m}(1+{||y||}_c)^\frac{m-2}{2}||f(y)||_cdy< \infty.
 \end{equation}
$\bullet L^\infty(\mathbb{R}^m)\otimes Cl_{0,m}$ the space  of essentially bounded functions on  $\mathbb{R}^m$ taking values in $Cl_{0,m}$ endowed with the norm
\begin{equation}
||f||_{\infty,c}:=\text{inf}\{C \geq 0 ;\quad ||f(x)||_c\leq C \quad\text{for almost every }x\in\mathbb{R}^m\}.
\end{equation}

 \section{Clifford-Fourier Transform}

\begin{defin}
 Let $f \in  B(\mathbb{R}^{m})\otimes Cl_{0,m}$. {\it The Clifford-Fourier transform is given  by }  {\rm (see \cite{key-7})} :
  \begin{equation}
 \mathcal{F}_{\pm}(f)(y)={(2\pi)}^\frac{-m}{2}\int_{\mathbb{R}^m} K_{\pm}(x,y)f(x)dx,
  \end{equation}
 where 
  \begin{equation}
 K_{\pm}(x,y)=e^{\mp i\frac{\pi}{2}\Gamma_y}e^{-i<x,y>}.
  \end{equation}
 \end{defin}
 \begin{lemme}Let $m$ be even.  Then 
  \begin{equation}
{||K_{\pm}(x,y)||}_c\leq C e^{{||x||}_c{||y||}_c}, \qquad \forall x,y\in \mathbb{R}^m,
 \end{equation}
with 
$C$ a positive constant.
\end{lemme}
\begin{proof}
By proposition 3.4 of \cite{key-7}, it is enough to prove the lemma for $K_-(x,y)$.\\
\noindent For $m=2$, it is shown in  \cite[p.15]{key-7}, that  for $x=x_1e_1+x_2e_2$ and $y=y_1e_1+y_2e_2$, we have
$$K_-(x,y)=cos(x_1y_2-x_2y_1)+e_{12} sin(x_1y_2-x_2y_1).$$
\noindent Thus
$${||K_-(x,y)||}_c=1.$$
\noindent For $m>2$. Recall that the kernel $K_-(x,y)$, for $x, y \in \mathbb{R}^m$, can be decomposed as the following 
$$K_-(x,y)=K_0^-(x,y)+\displaystyle\sum_{i<j}e_{ij} K_{ij}^-(x,y) ,$$
with $K_0^-(x,y)$ and $K_{ij}^-(x,y)$  scalar functions,\\
\noindent  By lemma 5.2 in \cite{key-7}, for $x, y \in \mathbb{R}^m$,
$$|K_0^-(x,y)|\leq c \left(1+||x||_c||y||_c\right)^\frac{m-2}{2},$$
$$|K_{ij}^-(x,y)|\leq c \left(1+||x||_c||y||_c\right)^\frac{m-2}{2}.$$
Thus 
$$|K_-(x,y)|\leq C \left(1+||x||_c||y||_c\right)^\frac{m-2}{2}.$$
Since for $u\geq0$
$$(1+u)^n e^{-u}\leq \frac{n^n}{e^{n-1}}, \forall n\in \mathbb{N},$$
we conclude.

\end{proof} 
 \begin{lemme} {\rm (see \cite{key-5} )} 
Let $c>0$  and $f\in B(\mathbb{R}^m)\otimes Cl_{0,m}$, then 
$$K_{\pm}(x,cy)= K_{\pm}(cx,y),\quad \forall x, y \in \mathbb{R}^m.$$
\noindent  Assume that $f_c(x):=f(cx),$   $x\in\mathbb{R}^m$. Then 
\begin{equation}
 \mathcal{F}_{\pm} (f_c)(\lambda)=c^m\mathcal{F}_{\pm} (f)(c^{-1}\lambda).
\end{equation}
\end{lemme}

\begin{theorem}{\rm(see \cite{key-6})}\\
i) The Clifford-Fourier transform defines a continuous operator mapping  $S(\mathbb{R}^m)\otimes Cl_{0,m}$ to $S(\mathbb{R}^m)\otimes Cl_{0,m}$ \rm{(see \cite[Theorem 6.3]{key-6})}.\\ In particular,  when $m$ is even,  we have 
$$\mathcal{F}_{\pm}\mathcal{F}_{\pm}=id_{S(\mathbb{R}^m)\otimes Cl_{0,m}}.$$
ii) The Clifford-Fourier transform extends from $S(\mathbb{R}^m)\otimes Cl_{0,m}$ to a continuous map on $L^2(\mathbb{R}^m)\otimes Cl_{0,m}$.\\ In particular, when $m$ is even, we have 
$$||\mathcal{F}_{\pm}(f)||_{2,c}=||f||_{2,c},$$
$\text{for all}\quad f\in L^2(\mathbb{R}^m)\otimes Cl_{0,m}$ \rm{(see \cite[Theorem 6.4]{key-6})}.
\end{theorem}
\begin{theorem} Let $\delta>0$ and $P\in \mathcal{P}_k(\mathbb{R}^m)$. Then, there exists $Q\in \mathcal{P}_k(\mathbb{R}^m)$ which satistifies :
\begin{equation}
\mathcal{F}_{\pm}(P(.)e^{-\delta||.||_c^2})(x)=Q(x)e^{-\frac{||x||_c^2}{4\delta}}.
\end{equation}

\end{theorem}
\begin{proof}

Considering the basis  of $L^2(\mathbb{R}^m)\otimes Cl_{0,m}$  
 
$$\phi_{k,l,j}(x)= e^{-\frac{||x||_c^2}{2}}H_{k,m,l}(\sqrt{2}x)P_l^{(j)}(\sqrt{2}x),$$ 

 \noindent with $k,l \in \mathbb{N}\cup\{0\}$ and $j=1,.., dim(\mathcal{M}_n^+(l))$.\\
   $H_{k,m,l}$  are the generalized Clifford-Hermite polynomials and 
$$\left\{P_l^{(j)}, j=1,2,...,dim(\mathcal{M}_n^+(l)\right\}$$
denotes an orthogonal basis of $\mathcal{M}_n^+(l)$ where $\mathcal{M}_n^+(l)$ the set of all left inner spherical monogenics of degree $l$. This basis constitutes eigenfunction of the Clifford-Fourier transform (see \cite{key-2}). Note $H_{k,m,l}$ is a polynomial of deree $s$ in the variable $x $ with the real coefficients depending  on $k$.
Thus, by unicity of eigenfunction of the Clifford-Fourier transform,  for every $P \in \mathcal{P}_k(\mathbb{R}^m)$ there exists  a unique polynomial $Q$ of degree $k$ such that  
$$\mathcal {F}_{\pm}(P(.)e^{-\frac{||.||_c^2}{2} })(x)= Q(x) e^{-\frac{||x||_c^2}{2} },$$ with degree $Q$ is $k$ .\\
\noindent The proof is completed by lemma 3.2.
\end{proof}
\section{Clifford-Heat kernel}
In this section,  we introduce the Heat kernel  in Clifford analysis. Then,   we establish some properties of the Clifford-Heat kernel.
\begin{defin} We define the Clifford-heat kernel by 
\begin{equation}
N_c(x,s):=\frac{1}{(2\pi)^{\frac{m}{2}} (2s)^{\frac{m}{2}}} e^{-\frac{||x||_c^2}{4s}},\qquad \forall x \in \mathbb{R}^m,s>0,
\end{equation}
associated with the Clifford-Laplace operator $\Delta_x$.
\end{defin} 
\begin{theorem}
Let $x\in\mathbb{R}^m$ and $s>0$. Then $N_c(x,s)$ satisfies : 
\begin{equation}
\frac{\partial N_c(x,s)}{\partial s}-\Delta_xN_c(x,s)=0.
\end{equation}
\end{theorem}
\begin{proof}
On one hand, applying the Clifford-Dirac operator to the Clifford-Heat kernel, we obtain 
$$\displaystyle\partial_x N_c(x,s)=\frac{1}{(2\pi)^{\frac{m}{2}} (2s)^{\frac{m}{2}}} \partial_xe^{-\frac{||x||_c^2}{4s}}$$
$$\qquad\qquad\qquad\qquad\quad=-\frac{1}{(2\pi)^{\frac{m}{2}} (2s)^{\frac{m}{2}+1}}\sum_{i=1}^m e_i x_ie^{-\frac{||x||_c^2}{4s}} $$
$$\qquad\qquad\quad\quad=-\frac{1}{(2\pi)^{\frac{m}{2}} (2s)^{\frac{m}{2}+1}} xe^{-\frac{||x||_c^2}{4s}} .$$
Thus 
$$\partial _x^2N_c(x,s)=-\frac{1}{(2\pi)^{\frac{m}{2}} (2s)^{\frac{m}{2}+1}}\partial_x( xe^{-\frac{||x||_c^2}{4s}})$$
%$$\qquad\qquad\qquad\qquad\qquad\qquad=-\frac{1}{(2\pi)^{\frac{m}{2}} (2s)^{\frac{m}{2}+1}}\left(\partial_x( x)e^{-\frac{||x||_c^2}{4s}}+ x\partial_xe^{-\frac{||x||_c^2}{4s}}\right)$$
$$\qquad\qquad\qquad\quad=-\frac{1}{(2\pi)^{\frac{m}{2}} (2s)^{\frac{m}{2}+1}}(-m-\frac{x^2}{2s})e^{-\frac{||x||_c^2}{4s}}.$$
By  (2.6) and (2.9), we get
$$\Delta_x N_c(x,s)=\frac{1}{(2\pi)^{\frac{m}{2}} (2s)^{\frac{m}{2}+1}}(-m+\frac{||x||_c^2}{2s})e^{-\frac{||x||_c^2}{4s}}.$$
On the other hand, we have 
$$\partial_sN_c(x,s)=\frac{1}{(2\pi)^{\frac{m}{2}}} \partial_s\left(\frac{1}{(2s)^{\frac{m}{2}}}e^{-\frac{||x||_c^2}{4s}}\right)$$
$$\qquad\qquad\qquad\qquad=\frac{1}{(2\pi)^{\frac{m}{2}} (2s)^{\frac{m}{2}+1}}(-m+\frac{||x||_c^2}{2s})e^{-\frac{||x||_c^2}{4s}}$$
$$=\Delta_c N_x(x,s).\qquad$$
\end{proof}
\begin{theorem} The Clifford-Heat kernel satisfies the following  properties :\\
{\rm i)} For all $x\in\mathbb{R}^m$ and $ s>0$,
	\begin{equation} 
	\mathcal{F}_{\pm}(N_c(.,s))(x)=\frac{1}{(2\pi)^{\frac{m}{2}}} e^{-s||x||_c^2}.
	\end{equation}
{\rm ii)}	 Let $m$ be even. For $x\in\mathbb{R}^m$ and $ s>0$,
	\begin{equation}
	N_c(x,s)=\frac{1}{(2\pi)^{\frac{m}{2}}} \displaystyle\int_{\mathbb{R}^m}K_+(y,x)e^{-s||y||_c^2}dy.
	\end{equation}
{\rm iii)}	 For all $\lambda >0$, $x\in \mathbb{R}^m$ and  $s>0$,
	
	\begin{equation}
	N_c(\lambda^{\frac{1}{2}}x,\lambda s)=\lambda^{-\frac{m}{2}}N_c(x,s).
	\end{equation}	
{\rm iv)} For $s>0$,
	\begin{equation}
	||N_c(.,s)||_{1,c}=1.	
	\end{equation}
{\rm v)} For all $s,t>0$ and  $x\in \mathbb{R}^m$,
\begin{equation}
	N_c(.,t)\ast_{Cl} N_c(.,s)(x)=(2\pi)^\frac{-m}{2} N_c(x, s+t),
	\end{equation}
where $\ast_{Cl}$ denotes the Clifford-Fourier convolution \rm{(see \cite{key-7})}.
\end{theorem}
\begin{proof} 

{\rm	i)} One has 
	 $$\mathcal{F}_{\pm}(N_c(.,s))(x)=\frac{1}{(2\pi)^{\frac{m}{2}}}\displaystyle\int_{\mathbb{R}^m}K_{\pm}(y,x)N_c(y,s)dy$$
	 $$\qquad\qquad\qquad\qquad\quad=\frac{1}{{(2\pi)}^{m} (2s)^\frac{m}{2}}\displaystyle\int_{\mathbb{R}^m}K_{\pm}(y,x)e^{-\frac{||y||_c^2}{4s}}dy.$$
	 By a  change of variable, we get 
	 $$\mathcal{F}_{\pm}(N_c(.,s))(x)=\frac{1}{(2\pi)^{\frac{m}{2}} (2s)^\frac{m}{2}}\displaystyle\int_{\mathbb{R}^m}K_{\pm}(\sqrt{2s}z,x) e^{-\frac{||z||_c^2}{2}}dz.$$
	 Lemma 3.2 implies that
	 $$\mathcal{F}_{\pm}(N_c(.,s))(x)=\frac{1}{(2\pi)^{\frac{m}{2}} }\displaystyle\int_{\mathbb{R}^m}K_{\pm}(z,\sqrt{2s}x) e^{-\frac{||z||_c^2}{2}}dz$$
	 $$=\mathcal{F}_{\pm}(\phi)(\sqrt{2s}x),\quad$$
	 \noindent with $\phi(x)=e^{-\frac{||x||_c^2}{2}}$. Since $\mathcal{F}_{\pm}(\phi)(x)=\phi(x)$  \rm{(see \cite{key-7})}, we deduce the result.

\noindent {\rm ii)} Using  Theorem 3.3  and (4.3), we obtain 
	$$N_c(x,s)=\mathcal{F}_{\pm}\circ\mathcal{F}_{\pm}(N_c(x,s))=\frac{1}{(2\pi)^{\frac{m}{2}}}\displaystyle\int_{\mathbb{R}^m}K_{\pm}(y,x)\mathcal{F}_{\pm}(N_c(y,s))dy$$
	$$\qquad\qquad\qquad\qquad\qquad\quad=\frac{1}{(2\pi)^{\frac{m}{2}}}\displaystyle\int_{\mathbb{R}^m}K_{\pm}(y,x)e^{-s||y||_c^2}dy.$$
	
\noindent {\rm iii)} Since  $$N_c(\lambda^\frac{1}{2}x,\lambda s)=\frac{1}{(2\pi)^{\frac{m}{2}} (2\lambda s)^{\frac{m}{2}}} e^{-\frac{||\lambda^\frac{1}{2}x||_c^2}{4\lambda s}},$$
	the result is obvious.\\
	\noindent {\rm iv)} By  (2.14)  $$||N_c(.,s)||_{1,c}=\displaystyle\int_{\mathbb{R}^m}||N_c(x,s)||_cdx$$
	$$\qquad\qquad\qquad\qquad\quad=\frac{1}{(2\pi)^{\frac{m}{2}} (2s)^{\frac{m}{2}}} \displaystyle\int_{\mathbb{R}^m}e^{-\frac{||x||_c^2}{4s}}dx.$$
	A change of variable yields the desired result.\\
	\noindent {\rm v)} Since $N_c (.,s)$ is a radial function, then the Clifford-convolution coincides with the classical convolution {\rm(see \cite{key-7})}. Thus 
	$$N_c(.,t)\ast_{Cl} N_c(.s)(x)=(2\pi)^{-\frac{3 m}{2}} 2^{-m} (st)^{-\frac{m}{2}}\displaystyle\int_{\mathbb{R}^m}e^{-\frac{||x-y||_c^2}{4t}}e^{-\frac{||y||_c^2}{4s}}dy.$$
	Note that 
	$$s||x-y||_c^2=s||x||_c^2+s||y||_c^2-2s<x,y>.$$
	
	$$N_c(.,t)\ast_{Cl} N_c(.,s)(x)=(2\pi)^{-\frac{3 m}{2}} 2^{-m}(st)^{-\frac{m}{2}}e^{-\frac{||x||_c^2}{4(s+t)}}\displaystyle\int_{\mathbb{R}^m}e^{-\frac{||y\sqrt{s+t}-\frac{s}{\sqrt{s+t}}x||_c^2}{4st}}dy.$$
It follows from	 a change of variable and (4.6) that
	$$N_c(.,t)\ast_{Cl} N_c(.,s)(x)=(2\pi)^{-\frac{3 m}{2}} 2^{-\frac{m}{2}}(2st)^{-\frac{m}{2}}e^{-\frac{||x||_c^2}{4(s+t)}} (s+t)^{-\frac{m}{2}}\displaystyle\int_{\mathbb{R}^m}e^{-\frac{||z||_c^2}{4st}}dz\quad$$
	$$\qquad\qquad\quad=(2\pi)^{- m}e^{-\frac{||x||_c^2}{4(s+t)}}(2(s+t))^{-\frac{m}{2}}\displaystyle\int_{\mathbb{R}^m}N_c(z,st)dz$$
	$$=(2\pi)^\frac{-m}{2}N_c(x,s+t).\qquad\qquad\quad$$
\end{proof}
 \section{Hardy's theorem}
 In this section, we give a generalization of  Hardy's theorem for the Clifford-Fourier transform.

\begin{theorem}
Assume that $m$ is even. Let  $p$ and $q$ be positive constants.
Suppose $f$ is a measurable function on  $\mathbb{R}^m $ such that : 
\begin{equation}
||f(x)||_c\leq C e^{-p||x||_c^2},\quad x \in \mathbb{R}^m
\end{equation}
and 
\begin{equation}
||\mathcal{F}_\pm(f)(\lambda)||_c\leq C e^{-q||\lambda||_c^2},\quad \lambda \in \mathbb{R}^m,
\end{equation}
for some positive constant $C.$ Then, three cases   can occur :\\

	\noindent i) If   $pq>\frac{1}{4},$ then $f=0.$\\
	 ii) If   $pq=\frac{1}{4},$ then $f(x)=A e^{-p||x||_c^2},$ where $A$ is a constant.\\
	iii) If   $pq<\frac{1}{4},$ then there are many functions satisfying the assumptions. 
\end{theorem}
\begin{proof}
The basis  $\{\psi_{j,k,l}\}$ defined  in \cite{key-7} by :
$$\psi_{2j,k,l}(x):=L_j^{\frac{m}{2}+k-1}({||x||}_c^2) M_k^{(l)}e^{-\frac{{||x||}_c^2}{2}},$$
$$ \psi_{2j+1,k,l}(x):=L_j^{\frac{m}{2}+k}({||x||}_c^2)x M_k^{(l)}e^{-\frac{{||x||}_c^2}{2}},$$

\noindent where $j,k \in \mathbb{N}$, $M_k^{(l)}\in \mathcal{M}_k;\quad l=1,.., dim \mathcal{M}_k$
  gives an infinite number of examples for iii).\\
It is well known that by scaling (lemme 3.2), we may assume $p=q$ without loss of generality. The proof of i) is 
a simple  deduction of ii).\\

\noindent Assume $p=q=\frac{1}{2}.$ Since $f$ satisfying (5.1), then $f\in B(\mathbb{R}^m)\otimes Cl_{0,m}$.\\
\noindent  Moreover, we have for $\lambda \in \mathbb{R}^m\otimes\mathbb{C}$

$$||\mathcal{F}_{\pm}(f)(\lambda)||_c\leq{(2\pi)}^\frac{-m}{2}\int_{\mathbb{R}^m} ||K_{\pm}(x,\lambda)||_c||f(x)||_cdx$$
$$\qquad\qquad\qquad\leq C{(2\pi)}^\frac{-m}{2}\displaystyle\int_{\mathbb{R}^m} ||K_{\pm}(x,\lambda)||_ce^{-\frac{||x||_c^2}{2}}dx.$$

\noindent Applying Lemma 3.1, we get :
$$||\mathcal{F}_{\pm}(f)(\lambda)||_c\leq C{(2\pi)}^\frac{-m}{2}\displaystyle\int_{\mathbb{R}^m} e^{||x||_c||\lambda||_c-\frac{||x||_c^2}{2}}dx$$
$$\qquad \qquad\qquad\qquad\leq e^{\frac{||\lambda||_c^2}{2}}C{(2\pi)}^\frac{-m}{2}\displaystyle\int_{\mathbb{R}^m}e^{\frac{-\left(||x||_c-||\lambda||_c\right)^2}{2}}dx.$$
Thus
\begin{equation}
||\mathcal{F}_{\pm}(f)(\lambda)||_c\leq Ce^{\frac{||\lambda||_c^2}{2}},
\end{equation}
where $C$ is a  positive constant.\\
As $\mathcal{F}_{\pm}(f)$ is an entire function verifying  (5.2) and (5.3), lemma 2.1 in \cite{key-17} allows to express $\mathcal{F}_{\pm}(f)$ as follows :
$$\mathcal{F}_{\pm}(f)(x)=A e^{-\frac{||x||_c^2}{2}},$$ with A is a constant.\\
Since $f$ satisfies (5.1), then $f\in L^2(\mathbb{R}^m)\otimes Cl_{0,m}$. By the inversion formula   (Theorem 3.3), the proof is completed.
\end{proof}
\section{Miyachi's theorem for the Clifford-Fourier transform}
In this section, we provide an analogue of Miyachi's theorem for the Clifford-Fourier transform.
\begin{lemme}Let $m$ be even. Suppose $f$ is a measurable function on $\mathbb{R}^m$ such that  
\begin{equation} 
e^{a||.||_c^2}f \in L^p(\mathbb{R}^m)\otimes Cl_{0,m}+L^q(\mathbb{R}^m)\otimes Cl_{0,m},
\end{equation}
for some $a>0$ and $1\leq p,q\leq +\infty$.

\noindent Then $\mathcal{F}_{\pm}(f )$  is well defined. Moreover, there exists $C > 0$ such that,
\begin{equation} 
||\mathcal{F}_{\pm}(f )(z)||_c \leq C e^{||z||_c^2/4a}, \qquad \forall z\in\mathbb{R}^m\otimes \mathbb{C}.
\end{equation}

\end{lemme}

\begin{proof}
By (6.1)  there exists $u \in L^p(\mathbb{R}^m)\otimes Cl_{0,m}$ and  $v \in L^q(\mathbb{R}^m)\otimes Cl_{0,m}$ sucht that 
$$f(x)=e^{-a||x||_c^2}\left(u(x)+v(x)\right).$$
Using H\rm{$\ddot{o} $}lder's inequality, one  has 
$$||f||_{B,c}\leq C (||u||_{p,c}+||v||_{q,c}).$$ \noindent Thus, $f\in B(\mathbb{R}^m)\otimes Cl_{0,m}$.
 
\noindent Lemma 3.1 yields for $z\in\mathbb{R}^m\otimes \mathbb{C}$
$$||K_{\pm}(x,z)||_c\leq e^{||x||_c||z||_c}$$
and
$$||\mathcal{F}_{\pm}(f )(z)||_c\leq C{(2\pi)}^\frac{-m}{2}\int_{\mathbb{R}^m}e^{{||x||_c}{||z||_c}} ||f(x)||_cdx$$
$$\qquad\qquad\qquad\qquad\qquad\qquad\quad\leq C{(2\pi)}^\frac{-m}{2}e^{\frac{||z||_c^2}{4a}}\displaystyle \int_{\mathbb{R}^m}e^{-a({||x||_c}-{\frac{||z||_c}{2a})^2}}e^{a||x||_c^2} ||f(x)||_cdx.$$
Applying H\rm{$\ddot{o} $}lder's inequality,  we deduce that
 $$||\mathcal{F}_{\pm}(f )(z)||_c\leq C{(2\pi)}^\frac{-m}{2}e^{\frac{||z||_c^2}{4a}} (||U||_{p,c}+||V||_{q,c}).$$

\noindent Thus, we are done.

\end{proof}
\begin{lemme}{\rm (see \cite{key-3})} Let $h$ be an entire function on $\mathbb{C}^m$ such that \\
$$||h(z)|| \leq Ae^{B||Rez||^2}$$ 
and
$$\int_{\mathbb{R}^m}log^+||h(y)|| dy < \infty,$$
for some positive constants $A$ and $B$ where $log^+ x = log$ $ x$ if $x > 1$ and $log^+ x = 0$ if $x \leq 1$. Then $h$ is a constant.
\end{lemme}

\begin{theorem}Let $m$ be even. Let $a,b$ and $ \lambda $ be positive constants and $p$, $q$  $\in[1,+\infty].$ Suppose that  $f$ is a measurable function  %\in B(\mathbb{R}^m)\otimes Cl_{0,m}$ 
such that
\begin{equation}
e^{a||.||_c^2}f \in L^p(\mathbb{R}^m)\otimes Cl_{0,m}+L^q(\mathbb{R}^m)\otimes Cl_{0,m}
\end{equation}
and 
\begin{equation}
\displaystyle\int_{\mathbb{R}^m}log^+\frac{||e^{b||y||_c^2}\mathcal{F}_{\pm}(f)(y)||_c}{\lambda}dy <\infty.
\end{equation}
Then, we have the following results : 
\begin{itemize}
	\item If $ab>\frac{1}{4}$, then $f=0$.
	\item If $ab=\frac{1}{4}$, then $f=CN_c(.,b)$ with $|C|\leq \lambda$.
	\item  If $ab<\frac{1}{4}$, then all functions $f$ written as $f=P(x)N_c(x,\delta)$ with $P\in\mathcal{P}_k$ and $\delta\in]b, \frac{1}{4a}[$ satisfy the assumptions of the theorem.
\end{itemize}

\end{theorem}
\begin{proof}
\noindent Put 
$$h(z)= e^{\frac{-z^2}{4a}}\mathcal{F}_{\pm}(f)(z),\quad \forall z\in\mathbb{R}^m\otimes\mathbb{C}.$$
\noindent On one hand,  we have for $z=\epsilon+i\eta\in \mathbb{R}^m\otimes\mathbb{C}$
$$z^2=\epsilon^2+i\epsilon\eta+i\eta\epsilon-\eta^2=-||\epsilon||_c^2-2i<\epsilon,\eta>+||\eta||_c^2,$$
and $$||z||_c^2=||\epsilon||_c^2+||\eta||_c^2.$$
Lemma 6.1 yields
\begin{equation}
||h(\epsilon+i\eta)||_c\leq Ce^{\frac{||\epsilon||_c^2}{2a}}.
\end{equation}

\noindent On the other hand, since  $log^+(cd)\leq log^+(c)+d$ for all $c,d >0$, for $ab>\frac{1}{4}$ it follows  that 
$$\displaystyle\int_{\mathbb{R}^m} log^+||h(y)||_cdy= \int_{\mathbb{R}^m} log^+||e^{\frac{-y^2}{4a}}\mathcal{F}_{\pm}(f)(y)||_cdy= \int_{\mathbb{R}^m} log^+||e^{\frac{||y||_c^2}{4a}}\mathcal{F}_{\pm}(f)(y)||_cdy$$
$$\qquad=\int_{\mathbb{R}^m} log^+\left(\frac{||e^{b||y||_c^2}\mathcal{F}_{\pm}(f)(y)||_c}{\lambda}\lambda e^{||y||_c^2{(\frac{1}{4a}}-b)}\right)dy$$
$$\qquad \qquad \qquad\qquad\leq \displaystyle\int_{\mathbb{R}^m} log^+\frac{||e^{b||y||_c^2}\mathcal{F}_{\pm}(f)(y)||_c}{\lambda}dy +\int_{\mathbb{R}^m}\lambda e^{||y||_c^2{(\frac{1}{4a}}-b)}dy<+\infty.$$
\noindent  As h is an entire function, applying  Lemma 6.2,  we get $h$ is  a constant.\\
Thus for $ab>\frac{1}{4}$, we have  
$$\mathcal{F}_{\pm}(f)(y)=Ce^{\frac{y^2}{4a}}=Ce^{\frac{-||y||_c^2}{4a}}, \forall y\in \mathbb{R}^m.$$
Subsquently, if $ab>\frac{1}{4}$ refering to (6.4), C must be zero.\\
It is clear that if $f$ satisfies (6.3), then $f\in L^2(\mathbb{R}^m)\otimes Cl_{0,m}$. Thus, Theorem 3.3 implies that  $$f=0.$$
\noindent If $ab =\frac{1}{4}$, then as in the previous case we have 
$$\displaystyle\int_{\mathbb{R}^m} log^+\frac{||h(y)||_c}{\lambda}dy=\displaystyle\int_{\mathbb{R}^m}log^+\frac{||e^{b||y||_c^2}\mathcal{F}_{\pm}(f)(y)||_c}{\lambda}dy<+\infty.$$
Thus, we deduce from  (6.5) and Lemma 6.2 that
$$\mathcal{F}_{\pm}(f)(y)=Ce^{-b||y||_c^2}.$$
 Under condition (6.4), we should have  $|c| \leq \lambda$.\\
Hence, by Theorem 3.3 and  (4.3), it follows that \\
$$f=CN(.,b),\text{with }|c| \leq \lambda.$$
\noindent For the last case ($ab<\frac{1}{4}$), suppose that $f= P(x)N_c(x,\delta)$ with $\delta\in]b, \frac{1}{4a}[$  and $P\in \mathcal{P}_k.$\\
\noindent Since $\delta <\frac{1}{4a},$ we get  $$e^{a||x||_c^2}f=e^{a||x||_c^2} \psi(x) e^{-\frac{||x||_c^2}{4\delta}}=\psi(x) e^{||x||_c^2(a-\frac{1}{4\delta})}\in L^p(\mathbb{R}^m)\otimes Cl_{0,m}+L^q(\mathbb{R}^m)\otimes Cl_{0,m},$$
where $\psi(x)=\frac{1}{(2\pi)^{\frac{m}{2}} (2\delta)^{\frac{m}{2}}}P(x)$.\\
\noindent  Using Theorem 3.4, we obtain
$$\mathcal{F}_{\pm}(f)(x)=Q(x)e^{-\delta||x||_c^2},$$
where $Q$ is a polynomial with the same degree of $P$.
Since $b<\delta$ we find that $$\int_{\mathbb{R}^m} log^+\frac{||e^{b||y||_c^2}\mathcal{F}_{\pm}(f)(y)||_c}{\lambda}dy=\int_{\mathbb{R}^m} log^+\frac{||e^{b||y||_c^2}Q(y)e^{-\delta||y||_c^2}||_c}{\lambda}dy<+\infty.$$

\end{proof}
\begin{corollary}
Let $f$ be a measurable function on $\mathbb{R}^m$ such that 
\begin{equation}
e^{a||.||_c^2}f \in L^p(\mathbb{R}^m)\otimes Cl_{0,m}+L^q(\mathbb{R}^m)\otimes Cl_{0,m}
\end{equation}
\begin{equation}
\displaystyle\int_{\mathbb{R}^m}||\mathcal{F}_{\pm}(f)(y)||_c^re^{rb||y||_c^2}dy <\infty,
\end{equation}
for some constants $a$, $b>0$, $1\leq p,q\leq +\infty$ and $0<r \leq\infty$.
\begin{itemize}
\item[(i)]If $ab\geq\frac{1}{4}$ then $f=0$.
\item[(ii)]If $ab<\frac{1}{4}$ then  all functions $f$ of  the form $f(x)=P(x)N_c(x,\delta)$ with $P\in\mathcal {P}_k$ and $\delta\in]b, \frac{1}{4a}[$ satisfy (6.6) and (6.7).
\end{itemize}
\end{corollary}
%\begin{remark}
%\end{remark} In (6.3) and (6.6)

\end{document}